\documentclass[twoside,leqno,twocolumn]{article}
\usepackage{ltexpprt}
\usepackage{tikz}
\usepackage{amsmath, amssymb}
\newcommand*\samethanks[1][\value{footnote}]{\footnotemark[#1]}

\DeclareMathOperator{\conv}{conv}
\DeclareMathOperator{\degen}{degen}

\begin{document}

\title{\Large On the Geometry and Extremal Properties of the Edge-Degeneracy
  Model\thanks{Partially supported by AFOSR grant \#FA9550-14-1-0141.}}
  \author{Nicolas Kim\footnote{Equal contribution}{ }\thanks{Department of
    Statistics, Carnegie Mellon University, Email: nicolask@stat.cmu.edu,
  arinaldo@cmu.edu.}
  \and
  Dane Wilburne\samethanks[2]{ }\thanks{Department of Applied Mathematics,
    Illinois Institute of Technology, Email: dwilburn@hawk.iit.edu,
  sonja.petrovic@iit.edu.}
  \and
  Sonja Petrovi\'c\samethanks
  \and
  Alessandro Rinaldo\samethanks[3]
}
\date{}

\maketitle


\begin{abstract} \small\baselineskip=9pt 
  The edge-degeneracy model is an exponential random graph model  that
  uses the graph degeneracy, a measure of the graph's connection density, and
  number of edges in a graph as its sufficient statistics. We show this model is
  relatively well-behaved by studying the statistical degeneracy of this model
  through the geometry of the  associated polytope. 

\end{abstract}

\section{Introduction}
    Statistical network analysis is concerned with developing statistical tools
    for assessing, validating and modeling the properties of random graphs, or
    networks. The very first step of any statistical analysis is the
    formalization of a statistical model, a collection of probability
    distributions over the space of graphs (usually, on a  fixed number of nodes
    $n$), which will serve as a reference model for any inferential tasks one
    may want to perform. Statistical models are in turn designed to be
    interpretable and, at the same time, to be capable of reproducing the
    network characteristics pertaining to the particular problem at hand.
    Exponential random graph models, or ERGMs, are arguably the most important
    class of models for networks with a long history. They are especially useful
    when one wants to construct models that resemble the observed network, but
    without the need to define an explicit network formation mechanism. In the
    interest of space, we single out classical references
    \cite{BarndorffNielsen}, \cite{Brown86}, \cite{Goodreau} and a recent review
    paper \cite{F-review}.

Central to the specification of an ERGM is the choice of sufficient statistic, a
function on the space of graphs, usually vector-valued,  that captures the
particular properties of a network that are of scientific interest.  Common
examples of sufficient statistics are the number of edges, triangles,  or
$k$-stars,  the degree sequence, etc; for an overview, see \cite{F-review}. The
choice of a sufficient statistic is not to be taken for granted: it depends on
the application at hand and at the same time it dictates the statistical and
mathematical behavior of the ERGM. While there is not a general classification
of `good' and `bad' network statistics, some lead to models that behave better
asymptotically than others, so that computation and inference on large networks
can  be handled in a reliable way. 

In an ERGM, the probability of observing any given graph depends on the graph
only through the value of its sufficient statistic, and  is therefore modulated
by how much or how little the graph expresses those properties captured by the
sufficient statistics.  As there is virtually no restriction on the choice of
the sufficient statistics, the class of ERGMs therefore possesses  remarkable
flexibility and expressive power,  and offers, at least in principle, a broadly
applicable and statistically sound means of validating any scientific theory on
real-life networks. However, despite their simplicity, ERGMs are also difficult
to analyze and are often thought to behave in pathological ways, e.g., give
significant mass to extreme graph configurations. Such properties are often
referred to as degeneracy; here we will refer to it as \emph{statistical
degeneracy} \cite{Handcock} (not to be confused with graph degeneracy below).
Further, their asymptotic properties are largely unknown, though there has been
some recent work in this direction; for example, \cite{ChattDiac:ergms} offer a
variation approach, while in some cases it has been shown that their geometric
properties can be exploited to reveal their extremal asymptotic behaviors
\cite{YRF:asymptotic}, see also \cite{Raz08}. These types of results are
interesting not only mathematically, but have statistical value: they provide a
catalogue of extremal behaviors as a function of the model parameters and
illustrate the extent to which statistical degeneracy may play a role in
inference.

In this article we define and study the properties of the ERGM whose sufficient
statistics vector consists of two quantities: the  edge count, familiar to and
often used in the ERGM family, and the graph degeneracy, novel to the statistics
literature. (These quantities may be scaled appropriately, for purpose of
asymptotic considerations; see Section~\ref{sec:EDergm}.)  As we will see, graph
degeneracy arises  from the graph's core structure, a  property   that is new to
the ERGM framework \cite{KARWA}, but  is a natural connectivity statistic that
gives a sense of how densely connected the most important actors in the network
are. The core structure of  a graph (see Definition~\ref{defn:coredegen}) is of
interest to social scientists and other researchers in a variety of
applications, including the identification and ranking of influencers (or
``spreaders") in networks (see \cite{Kitsak} and \cite{Bae}), examining robustness to node failure, and for
visualization techniques for large-scale networks \cite{Carmi}.  The degeneracy of a graph is
simply the statistic that  records the largest core. 

 Cores are used as descriptive statistics in several network applications
 (see, e.g., \cite{Pei}), but until recently, very little was known about  statistical
   inference from this type of graph property:  \cite{KARWA}  shows that cores
   are unrelated to node degrees and that restricting  graph degeneracy yields
   reasonable core-based ERGMs.  Yet, there are currently no rigorous
   statistical models for networks in terms of their degeneracy. The results in
   this paper thus add a dimension to our understating of cores by exhibiting
   the behavior of the joint edge-degeneracy statistic within the context of the
   ERGM that captures it, and provide extremal results critical to
   estimation and inference for the edge-degeneracy model. 

We define the edge-degeneracy ERGM in Section~\ref{sec:EDergm}, investigate its
geometric structure in Sections~\ref{sec:GeomOfPolytope},
\ref{sec:AsymptoticsAndFan}, and \ref{sec:asymptotics}, and summarize the
relevance to statistical inference in Section~\ref{sec:conclusion}.

\section{The edge-degeneracy (ED) model}\label{sec:EDergm} 

This section presents  the necessary graph-theoretical
tools, establishes notation, and introduces the ED model. 
Let $\mathcal{G}_n$ denote the space of (labeled, undirected) simple graphs on
$n$ nodes, so $\vert\mathcal{G}_n\vert=2^{\binom{n}{2}}$. 

To define the family of probability distributions over $\mathcal{G}_n$
comprising  the ED model, we first define the degeneracy statistic. 

\begin{Definition}  \label{defn:coredegen} 
Let $G=(V,E)$ be a simple, undirected graph.  The
  \emph{k-core} of $G$ is the maximal subgraph of $G$ with minimum degree at
  least $k$.  Equivalently, the $k$-core of $G$ is the subgraph obtained by
  iteratively deleting vertices of degree less than $k$. The \emph{graph
  degeneracy} of $G$, denoted $\degen(G)$, is the maximum value of $k$ for which
  the $k$-core of $G$ is non-empty. 
\end{Definition} 

This idea is illustrated in Figure~\ref{fig:core-example}, which shows a graph $G$ and its 2-core.  In this case, $\degen(G)=4.$

\begin{figure}[h]
\begin{tikzpicture}
  [scale=.2,auto=center,every node/.style={circle,fill=black},inner sep=1.5pt]
  \node (n1) at (0,0)  {};
  \node (n2) at (4,0)  {};
  \node (n3) at (2,3.5)  {};
  \node (n4) at (8,0)  {};
  \node (n5) at (10,3.5)  {};
  \node (n6) at (6,6)  {};
  \node (n7) at (-4,0)  {};
  \node (n8) at (-8,0)  {};
  \node (n9) at (-6,3)  {};
  \node (n10) at (0,6)  {};
  \node (n11) at (3,6)  {};
  \node (n12) at (-3,6)  {};
  \node (n13) at (0,9)  {};
  \node (n14) at (-2,3.5)  {};
  \foreach \from/\to/\weight in {n2/n5/1,n2/n3/1,n3/n4/1,n2/n4/1,n3/n5/1,n1/n2/1,n4/n5/1,n4/n6/1,n2/n6/1,n3/n6/1,n5/n6/1,n1/n7/1,n7/n8/1,n8/n9/1,n7/n9/1,n1/n10/1,n10/n11/1,n10/n12/1,n10/n13/1}
    \draw (\from) --(\to);
\end{tikzpicture}
\quad\quad
\begin{tikzpicture}
  [scale=.2,auto=center,every node/.style={circle,fill=black,inner sep=1.5pt}]
  \node (n1) at (0,0)  {};
  \node (n2) at (4,0)  {};
  \node (n3) at (2,3.5)  {};
  \node (n4) at (8,0)  {};
  \node (n5) at (10,3.5)  {};
  \node (n6) at (6,6)  {};
  \node (n7) at (-4,0)  {};
  \node (n8) at (-8,0)  {};
  \node (n9) at (-6,3)  {};

  \foreach \from/\to/\weight in {n2/n5/1,n2/n3/1,n3/n4/1,n2/n4/1,n3/n5/1,n1/n2/1,n4/n5/1,n4/n6/1,n2/n6/1,n3/n6/1,n5/n6/1,n1/n7/1,n7/n8/1,n8/n9/1,n7/n9/1}
    \draw (\from) --(\to);

\end{tikzpicture}
		\caption{A small graph $G$ (left) and its $2$-core (right).  The degeneracy of this graph is 4.}
		\label{fig:core-example}
\end{figure}
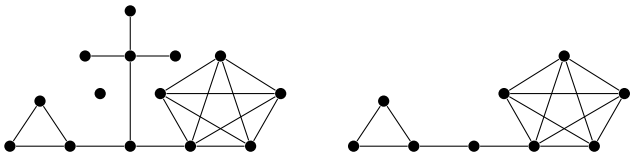

The \emph{edge-degeneracy} ERGM is the statistical model on $\mathcal{G}_n$ whose sufficient
statistics are the rescaled graph degeneracy and the edge count of the observed
graph. Concretely, for $G\in\mathcal{G}_n$ let
\begin{align}\label{eq:degen}
t(G)=\left( E(G)/{n \choose 2}, \degen(G)/(n-1) \right),
\end{align}
where $E(G)$ is the number of edges of $G$. 
The ED model on $\mathcal{G}_n$ is the ERGM $\{
P_{n,\theta}, \theta \in \mathbb{R}^2 \}$, where
\begin{align}\label{eq:model}
    P_{n,\theta}(G)=\exp\left\{\langle\theta,t(G)\rangle-\psi(\theta) \right\} 
\end{align}
is the probability of observing the graph $G \in  \mathcal{G}_n$ for the choice
of model parameter $\theta \in \mathbb{R}^2$. The log-partition function $\psi \colon
\mathbb{R}^2 \rightarrow \mathbb{R}$, given by $\psi(\theta) = \sum_{G \in
    \mathcal{G}_n} e^{\langle\theta,t(G)\rangle}$ serves as a normalizing
    constant, so that
    probabilities add up to $1$ for each choice of $\theta$ (notice that
    $\psi(\theta)< \infty$ for all $\theta$, as
    $\mathcal{G}_n$ is finite).

Notice that different choices of $\theta = (\theta_1,\theta_2)$ will lead to
rather different distributions. For example, for large and positive values of
$\theta_1$ and $\theta_2$ the probability mass concentrates on dense graphs,
while negative values of the parameters will favor sparse graphs. More
interestingly, when one parameter is positive and the other is negative, the
model will favor configurations in which the edge and degeneracy count will be
balanced against each other. Our results in Section \ref{sec:asymptotics} will
provide a catalogue of such behaviors in extremal cases and for large $n$.

    The normalization of the degeneracy and the edge count  in \eqref{eq:degen} and the presence of
    the coefficient ${n \choose 2}$ in the ED probabilities \eqref{eq:model} are to ensure a
    non-trivial limiting behavior as $n \rightarrow \infty$, since
    $E(G)$ and $\degen(G)$ scale differently in $n$ (see, e.g.,
    \cite{ChattDiac:ergms} and \cite{YRF:asymptotic}). This normalization is not strictly
    necessary for our theoretical results to hold.  
    However, the ED model, like most ERGMs, is not consistent, thus making asymptotic considerations somewhat problematic. 

 \begin{lemma}
  The edge-degeneracy model is an ERGM that is not consistent under sampling, as
  in \cite{SHALIZI}. 
\end{lemma}
\begin{proof}
  The range of graph degeneracy values when going from a graph with $n$ vertices to
  one with $n+1$ vertices depends on the original graph; e.g.\ if there is a
  2-star that is not a triangle in a graph with three vertices, the addition of
  another vertex can form a triangle and increase the graph degeneracy to 2, but if
  there was not a 2-star then there is no way to increase the graph degeneracy. Since
  the range is not constant, this ERGM is not consistent under sampling. 
\end{proof}
Thus, as the number of vertices $n$ grows, it is important to note the following property of the ED model, not uncommon in ERGMs: inference on the whole network cannot be done by applying the model to subnetworks.  

In the next few sections we will study the geometry of the ED model as a means
to derive some of its asymptotic properties. The use of polyhedral geometry in
the statistical analysis of discrete exponential families is well established:
see, e.g., \cite{BarndorffNielsen}, \cite{Brown86}, \cite{FieRin12:mle},
\cite{RFP13:mle}, \cite{RFZ09}.


\section{Geometry of the ED model polytope}\label{sec:GeomOfPolytope} 

The edge-degeneracy ERGM~\eqref{eq:model} is a discrete exponential family, for which the geometric structure of the model carries  important information about parameter estimation including  existence of maximum likelihood estimate (MLE) - see above mentioned references.  
This geometric structure is captured by the \emph{model polytope}.

The model polytope  $\mathcal{P}_n$ of the ED model on
$\mathcal{G}_n$  is the convex hull of the set of all possible edge-degeneracy pairs for graphs in $\mathcal{P}_n$.  In symbols, 
\begin{align*}
    \mathcal{P}_n:=\conv\Big\{ (E(G), \degen(G)), G \in \mathcal{G}_n
\Big\} \subset\mathbb{R}^2.
\end{align*}
  Note the use of  the unscaled version of the sufficient statistics in defining
  the model polytope.  In this section, the scaling used in model definition
  \eqref{eq:degen} has little impact on  shape of $\mathcal{P}_n$, thus - for
  simplicity of notation - we do not include it in the definition of
  $\mathcal{P}_n$. The scaling factors will be re-introduced, however, when we
  consider the normal fan and the asymptotics in
  Section~\ref{sec:AsymptoticsAndFan}. 

In the following,  we characterize the  geometric properties of $\mathcal{P}_n$
that are crucial to statistical inference. First, we arrive at a startling
result, Proposition~\ref{Prop:interior},  that every integer point in  the model
polytope is a realizable statistic.
 Second, Proposition~\ref{prop:proportion} implies that the observed network statistics will with high probability lie in the relative interior of the model polytope, which is an important property  because estimation algorithms are guaranteed to  behave well when off the boundary of the polytope.  
This also  implies that the MLE for the edge-degeneracy ERGM exists for many large graphs. In other words, there are very few network observations that can lead to statistical degeneracy, that is,  bad behavior of the model for which some ERGMs are famous. That behavior implies that a subset of the natural parameters is non-estimable, making complete inference impossible. Thus it being avoided by the edge-degeneracy ERGM is a desirable outcome. 
In summary,  Propositions~\ref{prop:proportion}, \ref{Prop:interior}, \ref{prop:lowerbdry} and Theorem~\ref{thm:maxlydegenerate}  completely characterize the geometry of $\mathcal{P}_n$ and thus solve  \cite[Problem 4.3]{Pet16} for this particular ERGM. Remarkably, this problem---although critical for our understanding of reliability of inference for such models-- has not been solved for most ERGMs except, for example, the beta model \cite{RFP13:mle}, which relied heavily on known graph-theoretic results. 

\smallskip
Let us consider $\mathcal{P}_n$  for some small values of $n$. The polytope
$\mathcal{P}_{10}$ is plotted in Figure \ref{Fig:poly}. 
\subsection*{The case $n=3$.}

There are four non-isomorphic graphs on 3 vertices, and each gives rise to a
distinct edge-degeneracy vector:  

\begin{center}
$
t\left( \begin{tikzpicture}
  [scale=.25,every node/.style={circle,fill=black,inner sep=1.5pt}]
  \node (n1) at (0,0)  {};
  \node (n2) at (4,0)  {};
  \node (n3) at (2,2)  {};
  \foreach \from/\to/\weight in {}
    \draw (\from) --(\to);
\end{tikzpicture}\right)=(0,0)
\ \ \ \ \ \ \ \ \ \ 
 t\left(\begin{tikzpicture}
  [scale=.25,every node/.style={circle,fill=black,inner sep=1.5pt}]
  \node (n1) at (0,0)  {};
  \node (n2) at (4,0)  {};
  \node (n3) at (2,2)  {};
  \foreach \from/\to/\weight in {n2/n3/1}
    \draw (\from) --(\to);
\end{tikzpicture}\right)=(1,1)
$
\end{center}
\begin{center}
$
 t\left(\begin{tikzpicture}
  [scale=.25,auto=center,every node/.style={circle,fill=black,inner sep=1.5pt}]
  \node (n1) at (0,0)  {};
  \node (n2) at (4,0)  {};
  \node (n3) at (2,2)  {};
  \foreach \from/\to/\weight in {n2/n3/1,n1/n3/1}
    \draw (\from) --(\to);
\end{tikzpicture}\right)=(2,1)
\ \ \ \ \ \ \ \ \ \
t\left(\begin{tikzpicture}
  [scale=.25,auto=center,every node/.style={circle,fill=black,inner sep=1.5pt}]
  \node (n1) at (0,0)  {};
  \node (n2) at (4,0)  {};
  \node (n3) at (2,2)  {};
  \foreach \from/\to/\weight in {n1/n2/1,n2/n3/1,n1/n3/1}
    \draw (\from) --(\to);
\end{tikzpicture}\right)=(3,2)
$
\end{center}
  Hence $\mathcal{P}_3=\conv{\{(0,0),(1,1),(2,1),(3,2)\}}$.  Note that in this
  case, each realizable edge-degeneracy vector lies on the boundary of the model
  polytope.  We will see below that $n=3$ is the unique value of $n$ for which
  there are no realizable edge-degeneracy vectors contained in the relative
  interior of $\mathcal{P}_n$.  

\subsection*{The case $n=4$.}

On 4 vertices there are 11 non-isomorphic graphs but only 8 distinct
edge-degeneracy vectors.  Without listing the graphs, the edge-degeneracy
vectors are:
$$(0,0),(1,1),(2,1),(3,1),(3,2),(4,2),(5,2),(6,3).$$ Here we pause to make the
simple observation that $\mathcal{P}_n\subset\mathcal{P}_{n+1}$ always holds.
Indeed, every realizable edge-degeneracy vector for graphs on $n$
vertices is also realizable for graphs on $n+1$ vertices, since adding a single
isolated vertex to a graph affects neither the number of edges nor the
graph degeneracy. 

\subsection*{The case $n=5$.}

There are 34 non-isomorphic graphs on $n=5$ vertices but only 15 realizable
edge-degeneracy vectors.  They are:  $$(0,0), (1,1), (2,1), (3,1), (3,2), (4,2),
(5,2), (6,3),$$ $$(4,1), (6,2), (7,2), (7,3), (8,3), (9,3), (10,4),$$ where the pairs listed on
the top row are contained in $\mathcal{P}_4$ and the pairs on the second row are
contained in $\mathcal{P}_5\setminus\mathcal{P}_4$.  
Here we make the observation that the proportion of realizable edge-degeneracy vectors
lying on the interior of the $\mathcal{P}_n$ seems to be increasing with $n$.  This phenomenon is addressed in Proposition~\ref{prop:proportion} below.
Figure~\ref{Fig:poly} depicts the integer points that define $\mathcal{P}_{10}$. 
    \begin{center}
    \begin{figure}
            \centering 
            \includegraphics[scale=0.42]{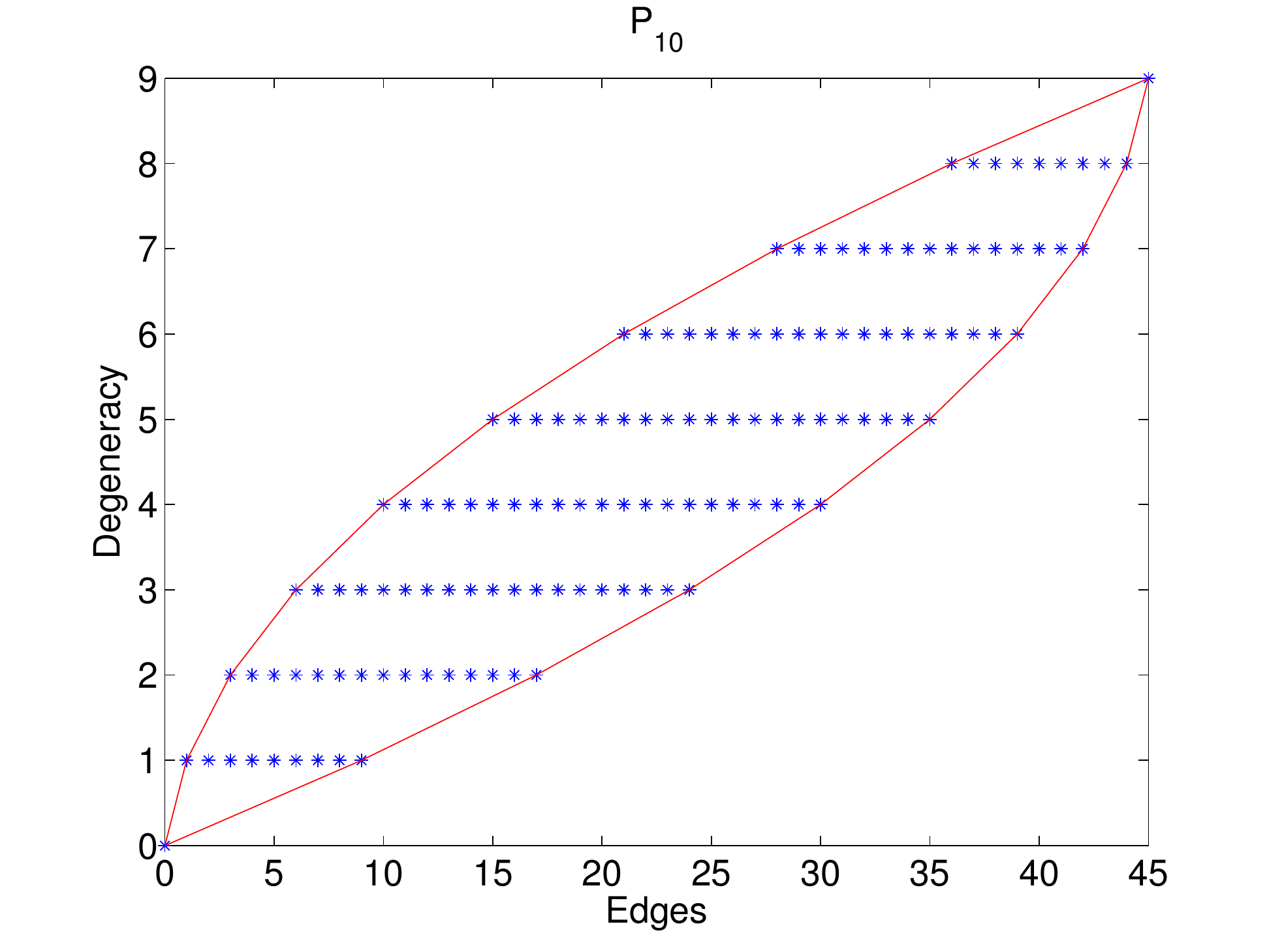}
            \caption[]%
            {{ The integer points that define the model polytope $\mathcal{P}_{10}$.}}    
    \label{Fig:poly}
    \end{figure}
    \end{center}
\subsection*{The case for general $n$.}

Many of the arguments below rely on the precise values of the coordinates of the extreme points of $\mathcal{P}_n.$

\begin{proposition} Let $U_n(d)$ be the minimum number of edges over all graphs on $n$ vertices with degeneracy $d$ and let $L_n(d)$ be the maximum number of edges over all such graphs.  Then, $$U_n(d)=\binom{d+1}{2}$$ and $$L_n(d)=\binom{d+1}{2}+(n-d-1)\cdot d.$$
\end{proposition}

\begin{proof} First, observe that if degen$(G)=d$, then there are at least $d+1$ vertices of
  $G$ in the $d$-core. Using this observation, it is not difficult to see that
  $U_n(d)=\binom{d+1}{2}$, since this is the minimum number of edges required to
  construct a graph with a non-empty $d$-core. Hence, the upper boundary of
  $\mathcal{P}_n$ consists of the points $(\binom{d+1}{2},d)$ for $0\le d\le
  n-1$.  For the value of $L_n(d)$, it is an immediate consequence of \cite[Proposition 11]{KARWA} that
    $L_n(d)=\binom{d+1}{2}+(n-d-1)\cdot d$.
\end{proof}
We use the notation $L_n(d)$ and $U_n(d)$ to signify that the extreme points $(L_n(d),d)$ and $(U_n(d),d)$ lie on the lower and upper boundaries of $\mathcal{P}_n,$ respectively.

It is well known in the theory of discrete exponential families that the MLE
exists if and only if the average sufficient statistic of the sample lies in the
relative interior of the model polytope.  This leads us to investigate which
pairs of integer points correspond to realizable edge-degeneracy vectors. 

\begin{proposition} Every integer point contained in $\mathcal{P}_n$ is a realizable
  edge-degeneracy vector.
\label{Prop:interior}
\end{proposition}

\begin{proof}  
  Suppose that $G$ is a graph on $n$ vertices with $\degen(G)=d\le n-1$.  Our strategy will be to
  show that for all $e$ such that $U_n(d)\le e\le L_n(d)$, there exists a graph
  $G$ on $n$ vertices such that degen$(G)=d$ and $E(G)=e$.

  It is clear that there is exactly one graph (up to
  isomorphism) corresponding to the edge-degeneracy vector $(\binom{d+1}{2},d)$;
  it is the graph
  	\begin{eqnarray}\label{Eqn:graph}
    	   K_{d+1}\cup \underbrace{K_1\cup\ldots\cup K_1}_{n-d-1\
      	   \textrm{times}},
  	\end{eqnarray} 
    i.e., the complete graph on $d+1$ vertices along
    with $n-d-1$ isolated vertices.
    
    Thus, for each $e$ such that
    $U_n(d)=\binom{d+1}{2}< e< \binom{d+1}{2}+(n-d-1)\cdot d=L_n(d)$, we must
    show how to construct a graph $G$ on $n$ vertices with graph degeneracy $d$
    and $e$ edges.  Call the resulting graph $G_{n,d,e}$.  To construct
    $G_{n,d,e}$, start with the graph in \ref{Eqn:graph}, which has degeneracy $d$.  Label the isolated
    vertices $v_1,\ldots,v_{n-d-1}$.  For each $j$ such that $1\le j\le
    e-\binom{d+1}{2}$, add the edge $e_j$ by making the vertex $v_i$ such that
    $i\equiv j\mod n-d-1$ adjacent to an arbitrary vertex of $K_{d+1}$.  This
    process results in a graph with exactly  $e=\binom{d+1}{2}+j$ edges, and
    since the vertices $v_1,\ldots,v_{n-d-1}$ still have degree at most $d$, our construction guarantees that we have not
    increased the graph degeneracy.  Hence, we have constructed $G_{n,d,e}$.
  
\end{proof}

The preceding proof also shows the following:

\begin{proposition}  $\mathcal{P}_n$ contains exactly
  \begin{eqnarray}\label{Eqn:int}
     \sum_{d=0}^{n-1}\left[(n-d-1)\cdot d+1\right]
  \end{eqnarray} integer points.  This is the number of realizable 
    edge-degeneracy vectors for every $n$.
\end{proposition}

The following property is useful throughout: 
\begin{lemma}\label{lem:symmetric}
  $\mathcal{P}_n$ is rotationally symmetric. 
\end{lemma}
\begin{proof}
  For $n \geq 3$ and $d \in \{1, \dots, n-1\}$, 
  \[
    L_n(d) - L_n(d-1) = U_n(n-d) - U_n(n-d-1). 
  \]
\end{proof}
Note that the center of rotation is the point $((n-1)n/4, (n-1)/2)$. The
rotation is 180 degrees around that point. 

As mentioned above, the following nice property of the ED model polytope-in conjunction with the partial characterization of the boundary graphs below- suggests that the MLE for the ED model exists for most large graphs. 
\begin{proposition} \label{prop:proportion}
  Let $p_n$ denote the proportion of realizable edge-degeneracy 
  vectors that lie on the relative interior of $\mathcal{P}_n$.  Then,
  \begin{align*}
  \lim_{n\to\infty}p_n=1.
  \end{align*}
\end{proposition}
\begin{proof}This result follows from analyzing the formula
  in~\ref{Eqn:int} and uses the following lemma: 
  \begin{lemma}
    There are $2n-2$ realizable lattice points on the boundary of
    $\mathcal{P}_n$, and each is a vertex of the polytope. 
  \end{lemma}
  \begin{proof}

    Since $\mathcal{P}_n \subseteq \mathbb{Z}^+ \times \mathbb{Z}^+$, and
    $(0,0) \in \mathcal{P}_n$ for all $n$, we know that $(0,0)$ must be a vertex
    of $\mathcal{P}_n$. By the rotational symmetry of $\mathcal{P}_n$, $(n-1,
    (n-1)n/2)$ must be a vertex, too.  It is clear that the points of the form $(U_n(d),d)$ and
    $(L_n(d),d)$ for $d\in\{1,\ldots,n-2\}$ are the only other points on the boundary; we will show
    that each of these points is in fact a vertex.  For this it suffices to observe that $U_n(d)$ is strictly
    concave and $L_n(d)$ is strictly convex as a function of $d$.  Hence, no interval $[U_n(d), L_n(d)]$
    is contained in the convex hull of any collection of intervals of the same form.  Thus, for each $d\in
    \{1,\ldots,n-2\}$ the points $(U_n(d),d)$ and $(L_n(d),d)$ are vertices of $\mathcal{P}_n$.

  \end{proof}

  To prove Proposition~\ref{prop:proportion}, we then compute: 
  \[
    p_n = \frac{\sum_{d=0}^{n-1}\left[(n-d-1)\cdot
    d+1\right]-(2n-2)}{\sum_{d=0}^{n-1}\left[(n-d-1)\cdot d+1\right]} \to 1.
  \]
\end{proof}

\subsection{Extremal graphs of $\mathcal{P}_n$.}

Now we turn our attention to the problem of identifying the graphs corresponding
to extreme points of $\mathcal{P}_n$.  Clearly, the boundary point $(0,0)$ is
uniquely attained by the empty graph $\bar{K}_n$ and the boundary point
$({n \choose 2},n-1)$ is uniquely attained by the compete graph $K_n$.  The proof of
Proposition~\ref{Prop:interior} shows that the unique graph corresponding to the
upper boundary point $(U_n(d),d)$ is a complete graph on $d+1$ vertices union
$n-d-1$ isolated vertices.  The lower boundary graphs are more complicated, but
the graphs corresponding to two of them are classified in the following
proposition.

\begin{proposition}  \label{prop:lowerbdry}
   The graphs corresponding to the lower boundary point $(L_n(1),1)$ of $\mathcal{P}_n$ 
   are exactly the trees on $n$ vertices.  The unique (up to isomorphism) graph corresponding to the lower boundary
   point $(L_n(n-2),n-2)$ is the complete graph on $n$ vertices minus an edge.  
\end{proposition}

\begin{proof}
    First we consider lower boundary graphs with edge-degeneracy vector $(L_n(1),1)$.   A graph with degeneracy 1 must be
    acyclic, since otherwise it would have degeneracy at least 2.  Hence, such a graph must be a forest.  However, if the forest is not connected, one could add an edge without increasing the degeneracy, and thus it must be a tree.  For the second statement, the complete graph minus one edge has the most edges among all non-complete graphs.  Any other graph has either larger degeneracy or fewer edges.
\end{proof}

The graphs corresponding to extreme points $(L_n(d),d)$ for $2\le d\le n-3$ are
called \emph{maximally d-degenerate} graphs and were studied extensively in
\cite{BICKLE}. Such graphs have many interesting properties, but are quite
difficult to fully classify or enumerate.

In the following theorem, we show that the lower boundary of $\mathcal{P}_n$ is like
a mirrored version of the upper boundary. This partially characterizes the remaining
maximally $d$-degenerate graphs. 
\begin{theorem}\label{thm:maxlydegenerate}
  Let $G(U_n(d))$ to be the unique (up to isomorphism) graph on $n$ nodes
  with degeneracy $d$ that has the minimum number of edges, given by $U_n(d)$.
  Similarly, let $G(L_n(d)) \subset \mathcal{G}_n$ be the set of graphs on $n$
  nodes with degeneracy $d$ that have the maximum number of edges, given by
  $L_n(d)$. Then, for all $d \in \{0, 1, \dots, n-1\}$, \[\overline{G(U_n(d))}
    \in G(L_n(n-d-1)),\] where $\overline{G(U_n(d))}$ denotes the graph complement of $G(U_n(d))$.
\end{theorem}
\begin{proof}
  As we know, $G(U_n(d)) = K_{d+1} \cup \overline{K_{n-d-1}}$. Taking the
  complement, 
  \begin{equation}\label{eq:complement}
    \overline{G(U_n(d))} = \overline{K_{d+1}} + K_{n-d-1},
  \end{equation} where $+$ denotes the \emph{graph join} operation.
  We only need to show that this has $L_n(n-d-1)$ edges and graph degeneracy
  $n-d -1$. 
  
  Since $G(U_n(d))$ has ${d + 1 \choose 2}$ edges,
  $\overline{G(U_n(d))}$ must have ${n \choose 2} - {d + 1 \choose 2} =
  L_n(n- d-1)$ edges. As for the graph degeneracy, $K_1+K_{n-d-1}=K_{n-d}$ is a subgraph of
  $\overline{G(U_n(d))}$.
  Therefore, $\degen(\overline{G(U_n(d))})
  \geq n-d-1$. However, $ \degen(\overline{G(U_n(d))}) < n-d$
  because there are $d+1$ vertices of degree $n-d-1$, and a non-empty $(n-d)$-core would
  require at least $n-d+1$ vertices of degree at least $n-d$.  Thus, $\degen(\overline{G(U_n(d))})=n-d-1$, 
  as desired.
\end{proof}

\section{Asymptotics of the ED model polytope and its normal fan}\label{sec:AsymptoticsAndFan} 

Since we will let $n \rightarrow \infty$ in this section, it will be necessary to rescale the
polytope $\mathcal{P}_n$ so that it is contained in $[0, 1]^2$, for
each $n$. Thus,  we divide the graph degeneracy parameter by $n-1$ and the edge parameter by
${n \choose 2}$, as we have already done in \eqref{eq:degen}.  

While this rescaling has little impact on 
shape of $\mathcal{P}_n$ discussed in Section~\ref{sec:GeomOfPolytope}, it does affect its normal fan, a key geometric object
that plays
a crucial role in our subsequent analysis. We describe the normal fan of the
normalized polytope next.
\begin{proposition}\label{prop:normal1}
  All of the perpendicular directions to the faces of $\mathcal{P}_n$ are:
  \[\{\pm(1, -m) : m \in \{1, 2, \dots, n-1\}\}.\] So, after normalizing, we get
  that the directions are \[\{\pm(1, -\frac{2}{\alpha n}) : \alpha \in
  \{1/(n-1), 2/(n-1), \dots, 1\}\}.\]
\end{proposition}
\begin{proof}
  The slopes of the line segments defining each face of $\mathcal{P}_n$ are
  $1/\Delta_U(d)=1/d$ in the unnormalized parametrization. To get the slopes of
  the normalized polytope, just multiply each slope by ${n \choose 2}/(n-1) =
  n/2$.
\end{proof}

Our next goal is to describe the limiting shape of the normalized model polytope
and its normal fan as $n \rightarrow \infty$. We first collect some simple facts
about the limiting behavior of the normalized graph degeneracy and edge count.

\begin{proposition}\label{prop:normal2}
  If $\alpha \in [0,1]$ such that $\alpha (n-1) \in \mathbb{N}$ (so that
  $\alpha$ parameterizes the normalized graph degeneracy), then \[\lim_{n \to \infty}
  \frac{U_n(\alpha(n-1))}{{n \choose 2}} = \alpha^2.\] Furthermore, due to the
  rotational symmetry of $\mathcal{P}_n$, \[\lim_{n \to \infty}
  \frac{L_n(\alpha(n-1))}{{n \choose 2}} = 1-(1-\alpha)^2.\]
\end{proposition}
\begin{proof}
By definition, $U_n(\alpha(n-1)) = \frac{\alpha^2}{2}n^2 + o(n)$. Hence,
\[\frac{U_n(\alpha(n-1))}{{n \choose 2}} = \frac{\frac{\alpha^2}{2}n^2 +
o(n)}{\frac{n(n-1)}{2}} \to \alpha^2.\]
\end{proof}

We can now proceed to describe the set limit corresponding to the sequence
$\{\mathcal{P}_n \}_n$ of model polytopes.  Here, the notion of set limit is the same as in \cite[Lemma 4.1]{YRF:asymptotic}. Let
\[
    \mathcal{P} = \mathrm{cl}\left\{ t \in \mathbb{R}^2 \colon t =t(G), G \in \mathcal{G}_n, n = 1,2,\ldots \right\}
\]
be the closure 
of the set of all possible realizable statistics \eqref{eq:degen} from the
model. Using Propositions \ref{prop:normal1} and
\ref{prop:normal2} we can characterize $\mathcal{P}$ as follows. 

\begin{lemma} 
    \begin{enumerate}
      \item $\mathcal{P}_n \subset \mathcal{P}$ for all $n$ and $\lim_n \mathcal{P}_n =
	    \mathcal{P}$.
	\item 
    Let $L$ and $U$ be functions from $ [0,1]$ into $ [0,1]$ given by 
    \[
      L(x) = 1-\sqrt{1-x} \quad \text{and} \quad U(x) = \sqrt{x}.
    \]
    Then,
    \[
	\mathcal{P} = \left\{ ( x,y) \in [0,1]^2  \colon L(x) \leq y \leq U(x) \right\}.
\]
    \end{enumerate}
\end{lemma}

\begin{proof}  
     If $\alpha\in[0,1]$ is such that $\alpha(n-1)\in\mathbb{N},$ then 
     $$\frac{U_n(\alpha(n-1))}{\binom{n}{2}}=\alpha^2+\frac{\alpha-\alpha^2}{n}\ge\alpha^2$$ for any finite 
     value of $n$.  Similarly, $$\frac{L_n(\alpha(n-1))}{\binom{n}{2}}=2\alpha-\alpha^2+\frac{\alpha^2-\alpha}{n}
     \le1-(1-\alpha)^2,$$ for any finite value of $n$.  Hence, for any $n$,
     $$\alpha^2\le\frac{U_n(\alpha(n-1))}{\binom{n}{2}}\le\frac{L_n(\alpha(n-1))}{\binom{n}{2}}
     \le1-(1-\alpha)^2.$$  Together with Propositions~\ref{prop:normal1} and
     \ref{prop:normal2}, this completes the proof.
     
\end{proof}
The convex set $\mathcal{P}$ is depicted at the top of Figure \ref{fig:graphons}.

In order to study the asymptotics of extremal properties of the ED model, the
final step is to describe all the normals to $\mathcal{P}$. As we will see in
the next section, these normals will correspond to different extremal behaviors
of the model. Towards this end, we define the following (closed, pointed) cones
\[
    \begin{array}{l}
    C_\emptyset = \mathrm{cone}\left\{ (1,-2), (-1,0) \right\}, \\
    C_\mathrm{complete} = \mathrm{cone}\left\{ (1,0), (-1,2) \right\}, \\
    C_\mathrm{U} = \mathrm{cone}\left\{ (-1,0), (-1,2)\right\},\\
    C_\mathrm{L} = \mathrm{cone}\left\{ (1,0), (1,-2) \right\},\\
\end{array}
\]
where, for $A \subset \mathbb{R}^2$, $\mathrm{cone}(A)$ denote the set of all conic (non-negative) combinations of the
elements in $A$.
It is clear that $C_\emptyset$ and $C_\mathrm{complete}$ are the normal fan to
the points $(0,0)$ and $(1,1)$ of $\mathcal{P}$. As for the other two cones, it is not hard to
see that the set of all normal rays to the edges of the upper, resp. lower,
boundary of
$\mathcal{P}_n$ for all $n$ are dense in $C_U$, resp. $C_L$.
As we will show in the next section, the regions $C_\emptyset$ and
$C_\mathrm{complete}$ indicate directions of statistical
degeneracy (for large $n$) towards the empty and complete graphs, respectively. On the other
hand, $C_U$ and $C_L$ contain directions of non-trivial convergence to extremal
configurations of maximal and minimal graph degeneracy. See Figure
\ref{fig:conjecture-fan} and the middle and lower
part of Figure \ref{fig:graphons}.

\begin{figure}[h!]
  \centering
  \includegraphics[scale=1]{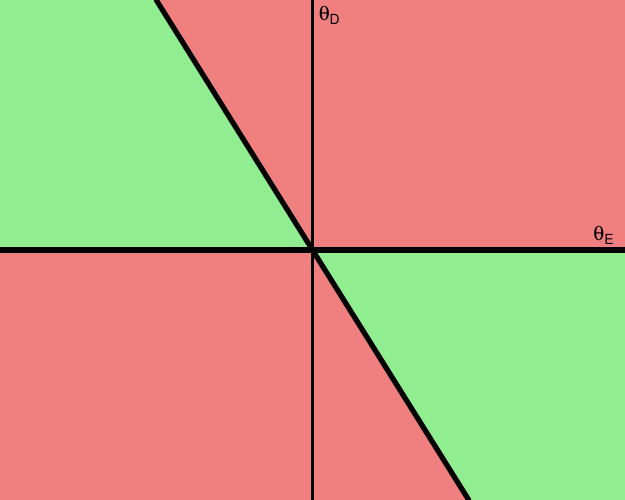}
  \caption{The green regions indicate directions of nontrivial convergence. The
    bottom-left and top-right red regions indicate directions towards the empty
    graph $\bar{K}_n$ and complete graph $K_n$, respectively.}
  \label{fig:conjecture-fan}
\end{figure}

\begin{figure}[h!]
  \centering
  \includegraphics[scale=0.65]{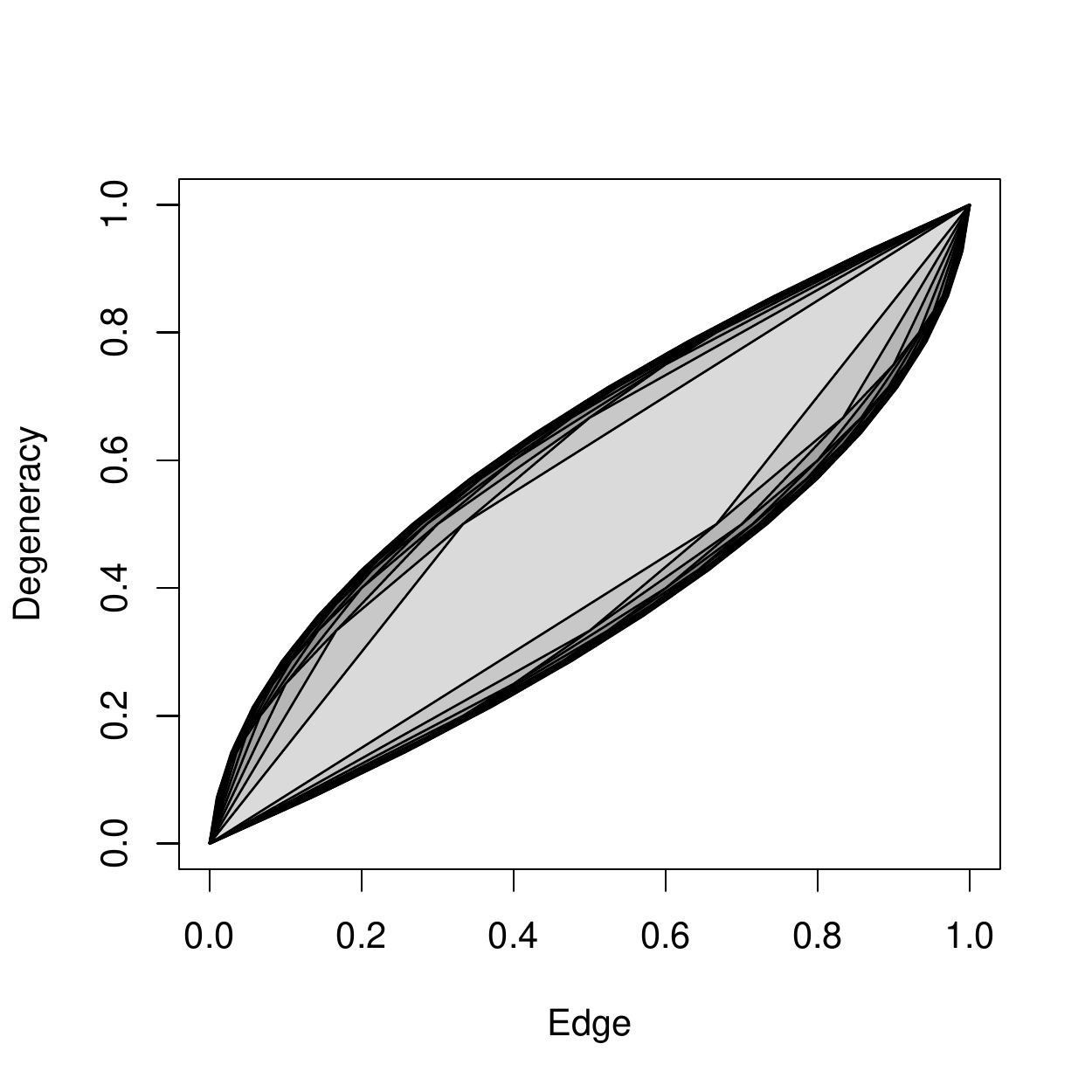} \\
  \includegraphics[scale=0.11]{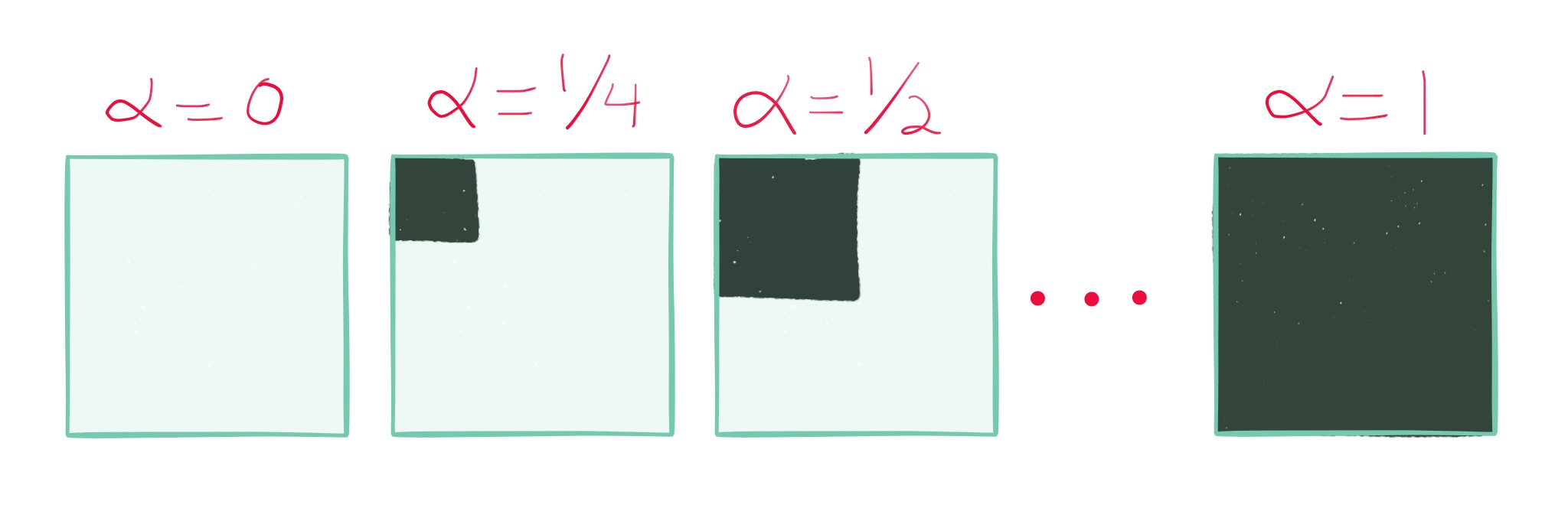} \\
  \includegraphics[scale=0.105]{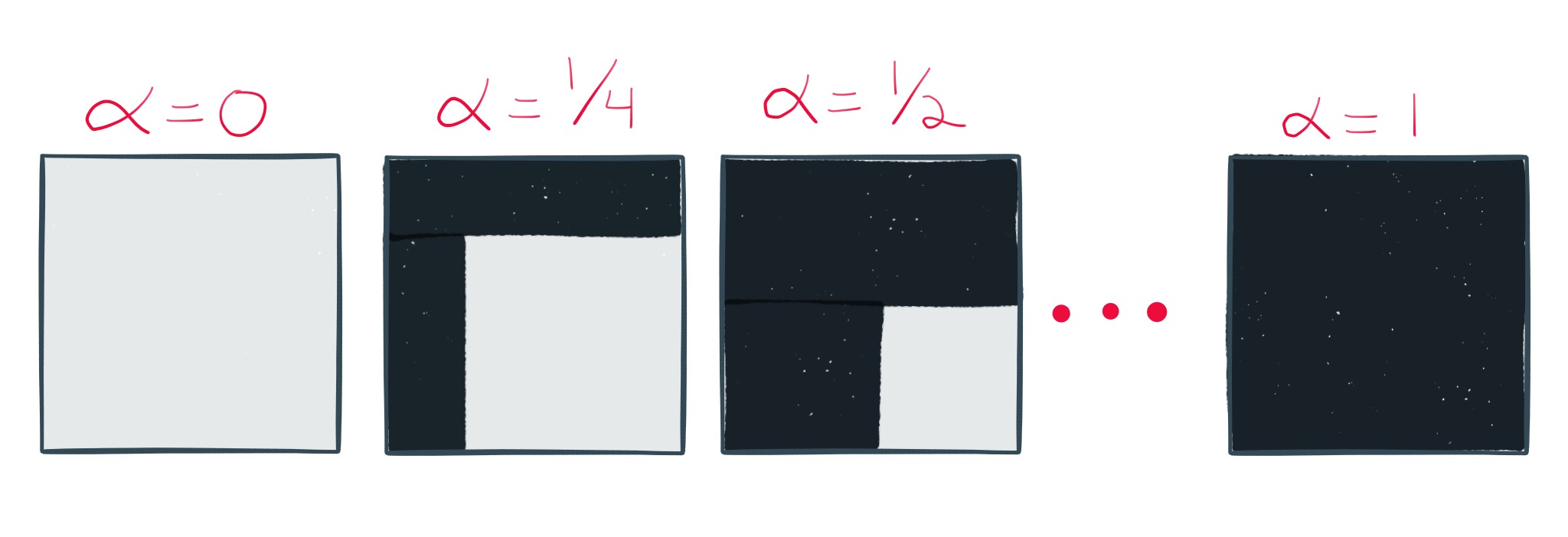}
  \caption{(Top) the sequence of normalized polytopes $\{\mathcal{P}_n\}_n$
  converges outwards, starting from $\mathcal{P}_3$ in the
  center. (Middle) and (bottom) are the representative infinite graphs along
  points on the upper and lower boundaries of $\mathcal{P}$, respectively,
  depicted as graphons \cite{Lovasz} for convenience.}
  \label{fig:graphons}
\end{figure}

\section{Asymptotical Extremal Properties of the ED Model}\label{sec:asymptotics}

In this section we will describe the behavior of distributions from the ED
model of the form
$P_{n, \beta + r d}$, where $d$ is a non-zero point in
$\mathbb{R}^2$ and $r$ a positive number. In particular, we will consider the
case in which $d$ and $\beta$ are fixed, but $n$ and  $r$ are large (especially $r$). We will
show that there are four possible types of extremal behavior of the model,
depending on $d$, the ``direction" along which the distribution becomes
extremal (for fixed $n$ and as $r$ grows unbounded). This dependence can be loosely expressed as follows: each $d$ will
identify one and only one value $\alpha(d)$ of the normalized edge-degeneracy
pairs such
that, for all $n$ and $r$
large enough, $P_{n, \beta + r d}$ will concentrate only on graphs whose normalized
edge-degeneracy value is arbitrarily close to $\alpha(d)$.

In order to state  this result precisely, we will first need to make some
elementary, yet crucial, geometric observations.  
Any $d \in \mathbb{R}^2$ defines a normal direction to one point on the
boundary of $\mathcal{P}$. Therefore, each $d \in \mathbb{R}^2$ identifies one point on
the boundary of $\mathcal{P}$, which we will denote $\alpha(d)$. Specifically, any $d \in C_\emptyset$ is in the normal cone to
the point $(0,0) \in \mathcal{P}$, so that $\alpha(d) = (0,0)$ (the normalized
edge-degeneracy of the
empty graph) for all $d \in C_\emptyset$
(and those points only).
Similarly, any $d \in C_{\mathrm{complete}}$ is in the normal cone to
the point $ (1,1) \in \mathcal{P}$, and therefore $\alpha(d) = (1,1)$ (the 
normalized
edge-degeneracy of $K_n$) for all $d \in
C_{\mathrm{complete}}$ (and those points only). On the other hand, if $d \in \mathrm{int}(C_{L})$, then $d$ is normal
to one point on the upper boundary of $\mathcal{P}$. Assuming
without loss of generality that $d=(1,a)$, $\alpha(d)$ is the point $(x,y)$ along the
curve $\{ L(x), x \in [0,1]\}$ such that $L'(x) = -\frac{1}{a}$. Notice that,
unlike the previous cases, if $d$ and $d'$ are distinct points in
$\mathrm{int}(C_{L})$ that are not collinear, $\alpha(d) \neq \alpha(d')$. 
Analogous
considerations hold for the points $d \in \mathrm{int}(C_U)$: non-collinear
points map to different points along the curve $\{ U(x), x \in [0,1] \}$. 

With these considerations in mind, we now present our main result about the
asymptotics of extremal properties of the ED model. 

\begin{theorem}
    \label{thm:extremal}
    Let $d \neq 0$ and consider the following cases.
\begin{itemize}
    \item $d \in \mathrm{int}(C_\emptyset)$.\\
	Then, for any $\beta \in \mathbb{R}^2$ and arbitrarily small $\epsilon \in (0,1)$ there exists a $
    	n(\epsilon)$ such that for all $n \geq n(\epsilon)$ there exists a  $r =
    	r(\epsilon,n)$ such that, for all $r \geq r(\epsilon,n)$ the empty graph has
    	probability at least $1 - \epsilon$ under $P_{n, \beta + r d}$.
    \item $d \in \mathrm{int}(C_{\mathrm{complete}})$.\\
	Then, for any $\beta \in \mathbb{R}^2$ and arbitrarily small $\epsilon \in (0,1)$ there exists a $
    	n(\epsilon)$ such that for all $n \geq n(\epsilon)$ there exists a  $r =
    	r(\epsilon,n)$ such that, for all $r \geq r(\epsilon,n)$ the complete  graph has
    	probability at least $1 - \epsilon$ under $P_{n, \beta + r d}$.
    \item  $d \in \mathrm{int}(C_L)$.\\
Then, for any $\beta \in \mathbb{R}^2$ and arbitrarily small $\epsilon,\eta \in (0,1)$ there exists a $
    	n(\epsilon)$ such that for all $n \geq n(\epsilon,\eta)$ there exists a  $r =
    	r(\epsilon,n)$ such that, for all $r \geq r(\epsilon,\eta,n)$  the set of
	graphs in $\mathcal{G}_n$ whose normalized edge-degeneracy is within $\eta$ of
	$\alpha(d)$ has
    	probability at least $1 - \epsilon$ under $P_{n, \beta + r d}$.
    \item  $d \in \mathrm{int}(C_U)$.\\
	Then, for any $\beta \in \mathbb{R}^2$ and arbitrarily small $\epsilon,\eta \in (0,1)$ there exists a $
    	n(\epsilon)$ such that for all $n \geq n(\epsilon,\eta)$ there exists a  $r =
    	r(\epsilon,n)$ such that, for all $r \geq r(\epsilon,\eta,n)$  the set of
	graphs in $\mathcal{G}_n$ whose normalized edge-degeneracy is within $\eta$ of
	$\alpha(d)$ has
    	probability at least $1 - \epsilon$ under $P_{n, \beta + r d}$.
\end{itemize}
\end{theorem}

\noindent {\it Remarks.} We point out that the directions along the boundaries 
of $C_\emptyset$ and $C_{\mathrm{complete}}$ are not part of our results. Our
analysis can also accommodate those cases, but in the interest of space, we omit
the results. More importantly, the value of $\beta$ does not play a role in
the limiting behavior we describe. We further remark that it is possible to
formulate a version of Theorem  
    \ref{thm:extremal} for each finite $n$, so that only $r$ varies. In that
    case, by Proposition \ref{prop:normal1}, for each $n$ there will only be $2(n-1)$ possible extremal
    configurations, aside from the empty and fully connected graphs. We have
    chosen instead to let $n$ vary, so that we could capture all possible cases.

\begin{proof}
  We only sketch the proof for the case $d \in \mathrm{int}(C_L)$, which
  follows easily  from the arguments in \cite{YRF:asymptotic}, in particular
  Propositions 7.2, 7.3 and Corollary 7.4. The proofs of the other cases are
  analogous.  First, we observe that the assumption (A1)-(A4) from \cite{RFZ09}
  hold for the ED model. 
  Next,   let $n$ be large enough such that $d$ is not in the normal cone
  corresponding to the points $(0,0)$ and $(1,1)$ of $\mathcal{P}_n$. Then, for each
  such $n$, $d$ either defines a direction corresponding to the normal of an
  edge, say $e_n$, of the upper boundary of $\mathcal{P}_n$ or $d$ is in the interior of
  the normal cone to a vertex, say $v_n$, of $\mathcal{P}_n$. Since
  $\mathcal{P}_n \rightarrow \mathcal{P}$,
  $n$ can be chosen large enough so that either the vertices of $e_n$ or $v_n$
  (depending on which one of the two cases we are facing) are within $\eta$ of
  $\alpha(d)$.

  Let us first consider the case when $d$ is normal to the edge $v_n$ of
  $\mathcal{P}_n$. Since every edge of $\mathcal{P}_n$ contains only two
  realizable pairs of normalized edge count and graph degeneracy, namely its
  endpoints, using the result in \cite{YRF:asymptotic}, one
  can choose $r = r(n,\epsilon,\eta)$ large enough so that at least $1-\epsilon$
  of the mass probability of $P_{n, \beta + r d}$ concentrates on the graphs in
  $\mathcal{G}_n$ whose normalized edge-degeneracy vector is either one of the
  two vertices in $e_n$. The claim follows from the fact that these vertices are
  within $\eta$ of $\alpha(q)$ For the other case in which $d$ is in the
  interior of the normal cone to the vertex $v_n$, again the results in
  \cite{YRF:asymptotic} yield that one can can choose $r = r(n,\epsilon,\eta)$
  large enough so that at least $1-\epsilon$ of the mass probability of $P_{n,
  \beta + r d}$ concentrates on graphs in $\mathcal{G}_k$ whose normalized
  edge-degeneracy vector is $v_n$. Since $v_n$ is within $\eta$ of $\alpha(q)$
  we are done. \end{proof}

The interpretation of Theorem \ref{thm:extremal} is as follows. If $d$ is a
non-zero direction in $C_\emptyset$ and $C_{\mathrm{complete}}$, then
$P_{n,\beta + rd}$ will exhibit statistical degeneracy regardless of
$\beta$ and for large enough $r$, in the
sense that it will concentrate on the empty and fully connected graphs, 
respectively. As shown in Figure \ref{fig:conjecture-fan},  $C_\emptyset$ and $C_{\mathrm{complete}}$ are fairly
large regions, so that one may in fact expect statistical degeneracy to occur
prominently when the model parameters have the same sign. The
extremal directions in $C_U$ and $C_L$ yields instead non-trivial behavior for
large $n$ and $r$. In this case, $P_{n, \beta + rd}$ will concentrate on
graph configurations  that are extremal in sense of exhibiting nearly maximal or
minimal graph degeneracy given the number of edges. 

Taken together, these results suggest that care is needed when fitting the ED
model, as statistical degeneracy appears to be likely.

    \section{Discussion}\label{sec:conclusion}
    
The goal of this paper is to introduce a new ERGM and demonstrate its statistical properties and asymptotic behavior captured by its geometry. 
The ED model  is  based on two graph statistics  that are not commonly used jointly and capture complementary information about the network: the number of edges and the graph degeneracy. The latter is extracted from  important information about the network's connectivity structure called cores and is often used as a descriptive statistic.  

The exponential family framework provides a beautiful connection between the
model geometry and its statistical behavior. To that end, we completely characterized the model polytope in Section~\ref{sec:GeomOfPolytope} for finite graphs and Section~\ref{sec:AsymptoticsAndFan} for the limiting case as $n\to\infty$.  
The most obvious implication of the structure of the ED model polytope is that the MLE exists for a significant proportion of large graphs. 
Another is that it simplifies greatly  the problem of projecting noisy data onto
the polytope and finding the nearest realizable point, as one need only worry
about the projection. Such projections play a critical role in data privacy 
problems, as they are used in computing a private estimator of the released data
with good statistical properties; see \cite{SesaVisheshBetaPrivacyAOS, SesaVisheshPrivateERGMs}. 
Finally, the structure of the polytope and its normal fan  reveal various extremal behaviors of the model, discussed in Section~\ref{sec:asymptotics}.

Note that the two statistics in the ED model summarize  very different properties of the observed graph, giving this seemingly simple model some expressive power and flexibility. 
 In graph-theoretic terms, the degeneracy summarizes the core structure of the graph, within which there can be few or many edges (see \cite{KARWA} for details); combining it with the number of edges produces   Erd\H{o}s-Renyi as a submodel. 
 
As discussed in Section~\ref{sec:EDergm}, different choices of the parameter vector, that is, values of the edge-degeneracy pair, lead to rather different distributions, from sparse to dense graphs as both parameters are negative or positive, respectively, as well as graphs where edge count and degeneracy are balanced for mixed-sign parameter vectors.  
Our results in Section~\ref{sec:asymptotics}  provide a catalogue of such behaviors in extremal cases and for large $n$.
The asymptotic properties we derive   offer interesting insights on the extremal
asymptotic behavior of the ED model. However, the asymptotic properties of
non-extremal cases, that is, those of distributions of the form $P_{n,\beta}$
for {\it fixed} $\beta$ and diverging $n$, remain completely unknown. While this
an exceedingly common issue with ERGMs, whose asymptotics are extremely difficult to describe, it would nonetheless be desirable to gain a better understanding of the ED model when the network is large.
In this regard, the variation approach put forward by \cite{ChattDiac:ergms}, which provides a way to resolve the asymptotics of ERGMs in general, may be an interesting direction to pursue in future work. 

\section*{Acknowledgements}

The authors are grateful to Johannes Rauh for his careful reading of an earlier draft of this manuscript and for providing many helpful comments.

\end{document}